\documentclass[11pt]{amsart}

\usepackage{amssymb}
\usepackage{amsmath}
\usepackage[all,cmtip]{xy}
\usepackage{enumitem}

\usepackage{amsfonts}
\usepackage{mathrsfs}
\usepackage{latexsym}
\usepackage{graphicx}
\usepackage{amscd,amssymb,amsmath,amsbsy,amsthm}
\usepackage[colorlinks,plainpages,urlcolor=blue]{hyperref}
\allowdisplaybreaks

\usepackage{tikz}
\usetikzlibrary{arrows,calc}
\tikzset{
>=stealth',
help lines/.style={dashed, thick},
axis/.style={<->},
important line/.style={thick},
connection/.style={thick, dotted},
}

\newtheorem{theorem}{Theorem}[section]
\newtheorem{lem}[theorem]{Lemma}
\newtheorem{cor}[theorem]{Corollary}

\theoremstyle{definition}
\newtheorem{dfn}[theorem]{Definition}
\newtheorem{example}[theorem]{Example}
\newtheorem{rem}[theorem]{Remark}

\newtheorem{assumption}[theorem]{Assumption}

\newtheorem{thm}{Theorem}

\newcommand{\bbZ}{\mathbb{Z}}

\newcommand{\calC}{\mathcal{C}}

\newcommand{\calI}{\mathcal{I}}

\newcommand{\frakI}{\mathfrak{I}}

\newcommand{\al}{\alpha}
\newcommand{\be}{\beta}
\newcommand{\de}{\delta}
\newcommand{\la}{\lambda}
\newcommand{\La}{\Lambda}

\newcommand{\hh}{\mathtt{h}} 

\newcommand{\unit}{\mathbf{1}} 

\newcommand{\pt}{{\operatorname{pt}}}

\DeclareMathOperator{\Hom}{Hom}   

\newcommand{\bfH}{\mathbf{H}}
\newcommand{\bfD}{\mathbf{D}}

\newcommand{\ep}{\epsilon}

\DeclareMathOperator{\trace}{\trace}

\newcommand{\mul}{{\operatorname{m}}}
\newcommand{\hy}{{\operatorname{h}}}

\DeclareMathOperator{\TL}{TL}

\title[Billey-type formula]
{Billey-type formula for KL-Schubert classes in hyperbolic cohomology}

\author[C.~Lenart]{Cristian~Lenart}
\address{State University of New York at Albany, 1400 Washington Avenue, Albany, NY 12222}
\email{clenart@albany.edu}

\author[K.~Zehr]{Krista~Zehr}
\email{kzehr@albany.edu}

\author[C.~Zhong]{Changlong~Zhong}
\email{czhong@albany.edu}

\begin{document}

\begin{abstract}This paper studies the KL-Schubert classes defined by Kazhdan-Lusztig bases in $K$-theory and hyperbolic cohomology of flag varieties. We first establish Poincar\'e dualities of these classes. We then focus on Grassmannians, and establish the Billey-type formula for KL-Schubert classes in hyperbolic cohomology. 
\end{abstract}
\maketitle

\section{Introduction}
Generalized Schubert calculus concerns classes determined by Schubert 
varieties in general cohomology theories $h$ beyond singular cohomology 
and $K$-theory. Since Schubert varieties are singular, these classes are 
usually defined by using the Bott-Samelson resolutions, and the 
definition usually depends on the choice of reduced sequences. 
There are two exceptions. The first is to consider the equivariant 
elliptic cohomology and define the classes depending on an additional 
parameter called the dynamical parameter. These classes were first 
defined by Rim\'anyi-Weber by using some early work of Borisov-Libgober 
on elliptic genus and it is related to Aganagic-Okounkov's elliptic 
stable envelop. See \cite{AO21, RW20, LZZ23} for more details. The other 
collection of classes are called Kazhdan-Lusztig (KL) Schubert classes 
defined for certain generalized cohomology theory, called hyperbolic 
cohomology. This cohomology theory is an example of the algebraic 
oriented cohomology of Levine and Morel. A geometric definition of 
hyperbolic cohomology is unknown (unless one uses the base-change from the algebraic cobordism). However, since there is a one-to-one 
correspondence between algebraic oriented cohomology theories and 
one-dimensional commutative formal group laws, the hyperbolic 
cohomology $\hh$  is the one defined by the so called hyperbolic formal 
group law
\[
F(x,y)=\frac{x+y-xy}{1-\mu^{-2} xy}, \quad \mu=t+t^{-1}\in 
R:=\bbZ[t,t^{-1}, \frac{1}{t+t^{-1}}]. 
\]

We now focus on $\hh_T(G/B)$. By using the twisted group algebra theory 
of Kostant-Kumar, for each formal group law,  one can define the formal 
affine Demazure algebra generated by the divided difference operators. 
For $\hh$ this is denoted by $\bfD_\hh$, and this algebra acts on 
$\hh_T(G/B)$.  The Bott-Samelson classes are generated by this 
action. It is shown in \cite{LZZ20} that there is an embedding of the 
Hecke algebra $\bfH$ into $\bfD_\hh$, so the KL-Schubert classes are 
defined by the action of the Kazhdan-Lusztig basis on $\hh_T(G/B)$.

These classes have many interesting properties. For example, if the 
Schubert variety is smooth, it is proved in \cite{LSZZ23} that the KL 
Schubert classes coincide with the Schubert classes. On the other hand, 
for type $A$ flag varieties, if there is a small resolution of the 
Schubert variety, then the KL-Schubert class coincide with the class 
defined by the small resolution. This shows that hyperbolic cohomology 
is closely related to singularities of Schubert varieties. 

An important formula in (equivariant) cohomology Schubert calculus is the so-called Billey formula~\cite{AJS94, B99}, which expresses combinatorially the localizations of equivariant Schubert classes at torus fixed points. This formula was extended to $K$-theory (for structure sheaves of Schubert varieties and their duals) in~\cite{G02, W04}. A Billey-type formula in hyperbolic cohomology was derived in~\cite{LZ17}. The present paper is concerned with a Billey-type formula for KL-Schubert classes. One important property that the Kazhdan-Lusztig basis $\{\gamma^\pm_w\}$ lacks is that of being multiplicatively generated; that is, the basis elements are not generated by the elements $\gamma^\pm_{s_i}$ indexed by simple reflections $s_i$. This multiplicative property is necessary to derive a Billey-type formula, so we encounter a major difficulty. Note here that there are two versions of Kazhdan-Lusztig bases, which are denoted by $\gamma^\pm_w$. See Section 2 below for their definitions. 

However, for type $A$ Grassmannians $G/P_J$, an early result of Fan-Green 
shows that by taking certain quotient of $\bfH$ that sends the basis $\gamma^\pm_w$ to zero for $w\in W$ that are not $iji$-avoiding,  the images of the remaining $\gamma^\pm_w$ will be 
generated by the images of $\gamma^\pm_{s_i}$ multiplicatively (see Lemma 
\ref{lem:FG}). Indeed this quotient defines the Temperley-Lieb algebra. This enables us to prove our first main result, Theorem 
\ref{thm:p}. 

It is well-known that  Schubert varieties and opposite Schubert varieties define Poincar\'e dual classes. Moreover, only classes determined by opposite Schubert 
varieties admit Billey-type formula. This includes the  opposite Schubert class and the SSM class of opposite Schubert cells in cohomology, and  the fundamental class, ideal sheaf class or the SMC  of opposite Schubert varieties (cells)   in $K$-theory.  These phenomena still hold for 
KL-Schubert classes. Roughly speaking, the two versions of Kazhdan-Lusztig bases correspond to Schubert varieties and opposite Schubert varieties. This suggests that it is essential to consider both versions of Kazhdan-Lusztig bases, and indeed they define Poincar\'e dual classes. For $K$-theory, this duality is proved in   Theorem  \ref{thm:Kdual} and for hyperbolic cohomology, it is proved in  Theorem \ref{thm:hyper} and  Theorem \ref{thm:hyperGP}. 

Combining these theorems, we obtain the Billey-type formula for KL-Schbuert class for the opposite Schubert variety (Theorem \ref{thm:Billey}) for the Grassmannian $G/P_J$. 

Note that there is another multiplicativity result of the 
Kazhdan-Lusztig basis, also in the type $A$ Grassmannian by Kirillov-Lascoux \cite{KL00}. This one does not behave well with the Weyl 
group, or more precisely, is not compatible with the process of writing 
$w$ as a reduced sequence. It is only compatible with small resolutions, see 
\cite{LSZZ23} for more details.

This paper is organized as follows: In Section 2 we recall the 
definition of the Kazhdan-Lusztig basis and establish the duality in $K$-theory
between the classes determined by the two versions of the 
Kazhdan-Lusztig bases  in Theorem \ref{thm:Kdual}. In Section 3 we focus on the Kazhdan-Lusztig 
basis for the type $A$ Grassmannian and prove Theorem \ref{thm:p} that indicates
 the multiplicativity of the Kazhdan-Lusztig basis. In 
Section 4 we turn to hyperbolic cohomology, establish the duality in 
this case in Theorem \ref{thm:hyper} and Theorem \ref{thm:hyperGP}. We then derive the Billey-type
formula in  Theorem \ref{thm:Billey} in Section 5, and provide some example of calculation. 

\noindent{\bf Acknowledgement} The authors would like to thank Rui Xiong for very helpful conversations, especially for pointing out to us the paper \cite{FG97}. C.L. was partially supported by the NSF grants DMS-1855592 and DMS-2401755. C.Z. was partially supported by Simons Foundation Travel Support for Mathematicians TSM-00013828. 
\section{Hecke algebra and $K$-theory KL-Schubert classes}
In this section, we recall the notion of formal group algebra and Hecke algebra, and establish a duality result of $K$-theory KL-Schubert classes.

\subsection{Kazhdan-Lusztig bases}
Let $G$ be a reductive group with Borel subgroup $B$ and  maximal split torus $T$ of rank $n$, with $\Lambda$ the group of characters of $T$. Let $\Pi=\{\al_1,...,\al_n\}$ be the set of simple roots, $\Sigma$ be the set of roots, and $W$ be the Weyl group with the longest element $w_0\in W$.  For each $J\subset\Pi$, we denote by $W_J$ the  subgroup of $W$ corresponding to $J$, and $\Sigma_J\subset \Sigma$ be the corresponding subset of roots, and $P_J$ be the corresponding parabolic subgroup.   Let $w_J\in W_J$ be the longest element. Let $W^J$ (resp. ${}^JW$) be the set of minimal length representative of left cosets (resp. right cosets). 

 Let $R=\bbZ[t, t^{-1}]$ and let $S^\mul$ be the group algebra $R[\Lambda]$.  Let $Q^\mul=Frac(S^\mul)$ and we define the corresponding twisted group algebras $Q^\mul_W=Q^\mul\rtimes R[W]$ with basis $\delta^\mul_w, w\in W$.  The product is given as follows:
\[
a\de_w^\mul\cdot  b\de_v^\mul=aw(b)\de^\mul_{w,v}, \quad a,b\in Q^\mul, w,v\in W.
\]
One can then define the Demazure-Lusztig operator  for each simple root $\al$
\[
\tau_\al=\frac{t^{-1}-t}{1-e^{-\al}}+\frac{t-t^{-1}e^{-\al}}{1-e^{-\al}}\delta_\al^\mul\in Q_W^\mul. 
\]
It satisfies the braid relations and the quadratic relation $\tau^2_\al=(t^{-1}-t)\tau_\al+1$. $\tau_{\al_i}, i=1,...,n$ generate the Hecke algebra $\bfH$. Denote \[\tau_i=\tau_{\al_i}, ~x_\al=1-e^{-\al}, ~x_i:=x_{\al_i}, ~t_w=t^{\ell(w)}, ~\ep_w=(-1)^{\ell(w)}. \]
For each $J\subset\Pi$, denote
\[
x_{J}^\mul =\prod_{\al\in \Sigma_J, \al<0}x_\al, \quad x^\mul _\Pi=\prod_{\al<0}x_\al=\prod_{\al>0}(1-e^{\al}), \quad x^\mul _{\Pi/J}=\frac{x^\mul _\Pi}{x^\mul _J}. \] 
Note that $w_0x_\Pi^\mul=\prod_{\al>0}x_\al$.

Let $P_{w,v}(t)\in \bbZ[t]$ be the Kazhdan-Lusztig polynomials, then there are two versions of Kazhdan-Lusztig bases:
\[
\gamma_v^+=\sum_{w\le v}t_vt_w^{-1}P_{w,v}(t^{-2})\tau_w, \quad \gamma_v^-=\sum_{w\le v}\ep_w\ep_vt_v^{-1}t_wP_{w,v}(t^2)\tau_w.
\]
For example,  for simple reflection $s_\al$, 
\[\gamma^+_\al:=\gamma^+_{s_\al}=\tau_\al+t=(\de_\al^\mul +1)\frac{t^{-1}-te^{-\al}}{1-e^{-\al}},\quad \gamma^-_\al:=\gamma^-_{s_\al}=\tau_\al-t^{-1}=\frac{t-t^{-1}e^{-\al}}{1-e^{-\al}}(\de_\al^\mul -1). \]
Moreover, it is known that $P_{w,w_J}=1$ for all $w\le w_J$, so 
\[
\gamma^+_{w_J}=\sum_{w\le w_J}t_{w_J}t_w^{-1}\tau_w, \quad \gamma^-_{w_J}=\sum_{w\le w_J}\ep_w\ep_{w_J}t_{w_J}^{-1}t_w\tau_w. 
\]
Denote $\gamma^\pm_j=\gamma^\pm_{\al_j}$. 
It is easy to show that 
$\gamma_j^-\gamma_j^+=0. $ This will play an key role in Theorem \ref{thm:p}. 

Note that  there is an anti-involution defined by 
\[
i:\bfH\to \bfH, \quad  \tau_w\mapsto \tau_{w^{-1}}, \quad t\mapsto t,
\]
and it maps $\gamma^\pm_w$ to $\gamma^\pm_{w^{-1}}$. Together with \cite[Corollary 5.18.(ii)]{LZZ20}, one can easily prove the following lemma. 
\begin{lem}\label{lem:KLfactor}
\begin{enumerate}
\item If $ws_\al<w$, then $\gamma_w^\pm=z^\pm \gamma_\al^\pm$ for some $z^\pm\in Q_W^\mul $. Similarly, if $s_\al w<w$, then  $\gamma_w^\pm=\gamma_\al^\pm z^\pm$ for some $z^\pm\in Q_W^\mul $. 
\item If $w=uw_J$ with $u\in W^J$, then $\gamma^\pm _w=z^\pm \gamma^\pm_{w_J}$ for some $z^\pm\in Q_W^\mul $. Similarly, if $w=w_Ju$ with $u\in {}^JW$, then $\gamma^\pm_w=\gamma^\pm_{w_J}z^\pm$ for some $z^\pm\in Q_W^\mul .$
\end{enumerate}
\end{lem}

\subsection{Complimentary property}
It is well known that  the coefficient of $\de^\mul _{w_0}$ in the product  
\[
\tau_w\tau_{u^{-1}w_0}^{-1}
\] is equal to 0 unless $u=w$ (see \cite[Lemma 4.6]{SZZ20}). We call it the complimentary property. Moreover, if $u=w$,  the coefficient of $\de^\mul _{w_0}$ in $\tau_w\tau_{u^{-1}w_0}^{-1}$ is equal to \[a_{w_0}:=\prod_{\al>0}\frac{t-t^{-1}e^{-\al}}{1-e^{-\al}}=\frac{\prod_{\al>0}(t-t^{-1}e^{-\al})}{w_0(x_\Pi^\mul)}.\]
This can be seen if one writes $\tau_w$ as a linear combination of $\delta^\mul_v, v\le w$, and just compute the leading coefficient. It turns out that this property is not unique to the DL operators. 

\begin{lem}\label{lem:KLcom}The two operators $\gamma^+_w,  \gamma^-_v$ are complimentary. That is, the coefficient of $\de_{w_0}^\mul $ in 
\[
\gamma_w^+\gamma^-_{u^{-1}w_0}
\] is equal to 0 unless $w=u$. 
\end{lem}

\begin{proof}
It is well-known that  
\begin{equation}\label{eq:Pdual} P_{u,v}=P_{u^{-1}, v^{-1}}, ~ \text{and }~
\sum_v\ep_u\ep_v P_{v,w}P_{w_0v, w_0u}=\de_{w,u}. 
\end{equation}
So we have 
\begin{align*} \gamma^-_w&=\sum_u \ep_{u}\ep_wt_{u}t_w^{-1}P_{u,w}(t^2)\tau_u\overset{\sharp_1}=\sum_{u}\ep_u\ep_w t_{w}t_u^{-1}P_{u,w}(t^{-2})\tau_{u^{-1}}^{-1}=\sum_u \ep_{w}\ep_ut_{w^{-1}}t_{u^{-1}}^{-1}P_{u^{-1}, w^{-1}}(t^{-2})\tau^{-1}_u. 
\end{align*}
Here $\sharp_1$ follows since $\gamma_w^-$ is invariant under Lusztig's involution. So 
\begin{align*}
\gamma^-_{w^{-1}w_0}&=\sum_u\ep_{w_0w}\ep_{u^{-1}}t_{w_0w}t_{u^{-1}}^{-1}P_{u^{-1}, w_0w}(t^{-2})\tau^{-1}_u.
\end{align*}
We use $z|_{w_0}$ to denote the coefficient of $\de^\mul _{w_0}$ in $z\in Q_W^\mul$. As mentioned above, we have 
\[
(\tau_w\tau_{u^{-1}w_0}^{-1})|_{w_0}=\de_{w,u}a_{w_0}. 
\]
Therefore, 
\begin{align*}
(\gamma_w^+\gamma^-_{u^{-1}w_0})|_{w_0}&=\left(\sum_{v_1}t_{w}t_{v_1}P_{v_1,w}(t^{-2})\tau_{v_1}\sum_{v_2}\ep_{w_0u}\ep_{v_2^{-1}}t_{w_0w}t_{v_2^{-1}}^{-1}P_{v_2^{-1}, w_0u}(t^{-2})\tau^{-1}_{v_2}\right)|_{w_0}\\
&=\left(\sum_{v}t_{w}t_v^{-1}P_{v,w}(t^{-2})\ep_{w_0u}\ep_{w_0v}t_{w_0w}t_{w_0v}^{-1}P_{w_0v, w_0u}(t^{-2})\tau_{v}\tau^{-1}_{v^{-1}w_0}\right)|_{w_0}\\
&=\sum_{v}P_{v,w}(t^{-2})\ep_{u}\ep_v P_{w_0v, w_0u}(t^{-2})a_{w_0}\\
&\overset{\eqref{eq:Pdual}}=\de_{w,u}a_{w_0}. 
  \end{align*}
\end{proof}

\subsection{$K$-theory KL Schubert classes}
To define $K$-theory classes, we  consider  the dual
\[
(Q_W^\mul)^*=\Hom_{Q^\mul}(Q_W^\mul, Q^\mul)=\Hom(W, Q^\mul),
\]
which is a $Q^\mul$-module with basis $f_w$ dual to the basis $\de_w^\mul$, and it has a product structure 
\begin{equation}\label{eq:Kprod}
f_wf_v=\de_{w,v}f_w. 
\end{equation}
The multiplicative unit is 
\[
\unit:=\sum_{v\in W}f_v. 
\]
There are two actions of $Q_W^\mul$ on $(Q_W^\mul)^*$, defined by 
\begin{equation}\label{eq:action}
a\de^\mul _w\bullet bf_v=bvw^{-1}(a)f_{vw^{-1}}, \quad a\de_w^\mul \odot bf_v=aw(b)f_{wv}, \quad a,b\in Q^\mul. 
\end{equation}
The two actions commute. The $\bullet$-action is $Q^\mul$-linear, but the $\odot$-action is not. The $\bullet$-action was considered in \cite{KK86, KK90, CZZ19, LZZ20, SZZ20}, and the $\odot$-action was considered in \cite{B97, K03, T09, MNS22}. In \cite{MNS22} they are called the right action and the left action, respectively. When we say a $W_J$-invariant element $f\in ((Q^\mul _w)^*)^{W_J}$, we mean
\[
\de_w^\mul \bullet f=f, \quad \forall w\in W_J. 
\]
It is easy to see that $((Q_W^\mul )^*)^{W_J}\cong Q^{\mul }\otimes_{S^{\mul }} K_T(G/P_J)$. Moreover,  $((Q_W^\mul)^*)^W$ is the algebraic model for $Q^\mul\otimes_{S^\mul}K_T(\pt)$, which is free $Q^\mul$-module with basis $\unit$. 

We fix a set $A$ of representatives of left cosets of $W/W_J$, and define 
\[
Y^\mul _J=\sum_{w\in W_J}\de_w^\mul \frac{1}{x^\mul _J}, \quad Y_{\Pi/J}^\mul =\sum_{w\in A}\de_{w}^\mul \frac{1}{x^\mul _{\Pi/J}}\in Q_W^\mul. 
\]
$Y_J^\mul \bullet\_:(Q^\mul_W)^*\to ((Q_W^\mul)^*)^{W_J}$ is the algebraic model for the pushforward $K_T(G/B)\to K_T( G/P)$, and $Y^\mul _{\Pi/J}\bullet\_$ is the one for the pushforward $K_T(G/P)\to K_T(\pt)$. Note that when applying $Y^\mul _{\Pi/J}\_\bullet$ on $((Q_W^\mul )^*)^{W_J}\cong Q^\mul\otimes K_T(G/P)$, the definition of $Y_{\Pi/J}^\mul $ will not depend on the choice of the left coset representatives. Moreover, it follows from \cite[Lemma 5.7]{CZZ19} that 
\begin{equation}\label{eq:comp}
Y_{\Pi/J}^\mul Y_J^\mul =Y_\Pi^\mul  . 
\end{equation}
There is the projection formula from \cite[Lemma 5.8]{CZZ19}, that is, 
\begin{equation}\label{eq:projection}
Y^\mul _J\bullet (fg)=fY^\mul _J(g), \quad g\in (Q^\mul _W)^*, f\in ((Q_W^\mul )^*)^{W_J}. 
\end{equation}

Denote 
\[
\pt_e=x_\Pi ^\mul \pt_e, ~\pt_w =\de_w^\mul\bullet\pt_e=wx_\Pi^\mul f_w.
\]
Here $\pt^\mul_w$ corresponds to the class in $K_T(G/B)$ determined the point $wB/B$. Recall the anti-involution from \cite[\S~13]{CZZ19}
\[
\iota:Q^\mul _W\to Q^\mul _W, ~p\de^\mul _w\mapsto \de^{\mul }_{w^{-1}}p\frac{w(x_\Pi^\mul )}{x^\mul _\Pi}, \quad p\in Q^\mul . 
\]
From \cite[\S~3]{LZZ20} one has 
\begin{equation}\label{eq:inv}
z\bullet \pt_e=\iota(z)\odot\pt_e, \quad Y^\mul _\Pi((z\bullet f)\cdot g)=Y^\mul _\Pi(f\cdot (\iota(z)\bullet g)), ~z\in Q_W^\mul, f,g\in (Q_W^\mul )^*. 
\end{equation}
Lastly, since $\iota(Y^\mul_J)=Y^\mul_J$, so  we have
\begin{equation}\label{eq:YJcom}
Y^\mul_J\odot\pt_e=Y^\mul_J\bullet \pt_e. 
\end{equation}

\begin{rem}All the formal properties concerning the product in $(Q^\mul_W)^*$, the two actions $\bullet, \odot$, and the involution $\iota$ hold if one replaces the multiplicative formal group law (e.g., $K$-theory) by any other formal group laws. See \cite{LZZ20} for more details. We will be using these properties later for hyperbolic cohomology $\hy$. 
\end{rem}

The following is our first main result on duality of $K$-theory KL Schubert classes. 
\begin{thm}\label{thm:Kdual}
\begin{enumerate}
\item We have the following duality in $(Q_W^\mul )^*\cong Q^\mul \otimes _{S^\mul} K_T(G/B)$:
\[
Y^\mul _\Pi\bullet \left((\gamma_w^-\odot\pt_e)\cdot( \gamma^+_{v^{-1}w_0}\bullet \pt_{w_0})\right)=\de_{w,v}\prod_{\al>0}(t-t^{-1}e^{-\al}) \unit. 
\]
\item The following two classes are dual in $((Q_W^\mul )^*)^{W_J}\cong Q^\mul \otimes _{S^\mul} K_T(G/P_J)$:
\[
Y^\mul _{\Pi/J}\bullet\left (Y_J^\mul \bullet (\gamma_w^-\odot\pt_e)\cdot (\gamma^+_{u^{-1}w_0}\bullet \pt_{w_0})\right)=\de_{w,u}\prod_{\al>0}(t-t^{-1}e^{-\al})\unit,\quad w,u\in W^J.
\]
\end{enumerate}
\end{thm}
\begin{proof}
(1). Note that 
\[
a_{w_0}=\frac{\prod_{\al>0}(t-t^{-1}e^{-\al})}{w_0(x_\Pi^\mul)}, 
\]
so 
\begin{equation}\label{eq:w0act}
a_{w_0}\de^\mul_{w_0}\bullet \pt_{w_0}=\prod_{\al>0}(t-t^{-1}e^{-\al})\frac{1}{w_0(x_\Pi^\mul)}\de_{w_0}^\mul\bullet w_0(x_\Pi^\mul)f_{w_0}\overset{\eqref{eq:action}}=\prod_{\al>0}(t-t^{-1}e^{-\al})f_e. 
\end{equation}
We have 
\begin{align*}
&Y_\Pi^\mul \bullet\left((\gamma_w^-\odot\pt_e)\cdot( \gamma^+_{u^{-1}w_0}\bullet \pt_{w_0})\right)\\
&\overset{\eqref{eq:inv}}=Y_\Pi^\mul \bullet\left((\iota(\gamma^-_{w^{-1}})\bullet\pt_e)\cdot ( \gamma^+_{u^{-1}w_0}\bullet \pt_{w_0})\right)\\
&\overset{\eqref{eq:inv}}=Y_\Pi^\mul \bullet \left(\pt_e\cdot (\gamma_w^-\bullet  \gamma^+_{u^{-1}w_0}\bullet \pt_{w_0})\right)\\
&=Y_\Pi^\mul \bullet \left(\pt_e\cdot (\gamma_w^-\gamma^+_{u^{-1}w_0}\bullet\pt_{w_0})\right)\\
&\overset{\eqref{eq:Kprod}}=Y_\Pi^\mul \bullet \left(\pt_e\cdot (\gamma_w^-\gamma^+_{u^{-1}w_0})|_{w_0}\de_{w_0}\bullet\pt_{w_0}\right)\\
&\overset{\text{Lem. }\ref{lem:KLcom}}=\de_{w,u}Y^\mul _\Pi\bullet \left(\pt_e\cdot (a_{w_0}\de^\mul_{w_0}\bullet \pt_{w_0})\right)\\
&\overset{\eqref{eq:w0act}}=\de_{w,u}\sum_{v\in W}\de^\mul_v\frac{1}{x_\Pi^\mul}\bullet \left(x_\Pi^\mul f_e\cdot {\prod_{\al>0}(t-t^{-1}e^{-\al})} f_e\right)\\
&=\de_{w,u}\prod_{\al>0}(t-t^{-1}e^{-\al})\unit. 
  \end{align*}

(2). Note that \[\ell(u^{-1}w_0)=\ell(w_J)+\ell(w_Ju^{-1}w_0)\] so by Lemma \ref{lem:KLfactor}, we have $\gamma^+_{u^{-1}w_0}=\gamma^+_{w_J}z$ for some $z\in Q^\mul_W$. From \cite[Proposition 4.10]{LZZ20} we know that \[\gamma_{w_J}^+=Y^\mul _J\prod_{\al>0}(t-t^{-1}e^\al),\]
so 
\[
\gamma^+_{u^{-1}w_0}\bullet \pt_{w_0}=Y_J^\mul \prod_{\al>0}(t-t^{-1}e^\al)\bullet z\bullet \pt_{w_0}\in Y_J^\mul \bullet (Q_W^\mul )^*=((Q_W^\mul )^*)^{W_J}. 
\]
It follows from the projection formula \eqref{eq:projection} that
\begin{equation}\label{eq:proj}
Y_J^\mul \bullet \left(g\cdot (\gamma^+_{u^{-1}w_0}\bullet \pt_{w_0})\right)=(Y_J^\mul \bullet g)\cdot (\gamma^+_{u^{-1}w_0}\bullet \pt_{w_0}), \quad g\in (Q_W^\mul )^*. 
\end{equation}
Therefore,  \begin{align*}
&Y^\mul _{\Pi/J}\bullet\left (Y^\mul _J\bullet (\gamma_w^-\odot\pt_e)\cdot (\gamma^+_{u^{-1}w_0}\bullet \pt_{w_0})\right)\\
&\overset{\eqref{eq:proj}}=Y^\mul _{\Pi/J}\bullet Y^\mul _J\bullet \left((\gamma^-_{w}\odot\pt_e)\cdot  (\gamma^+_{u^{-1}w_0}\bullet \pt_{w_0})\right)\\
&\overset{\eqref{eq:comp}}=Y_\Pi^\mul \bullet \left ((\gamma_w^-\odot\pt_e)\cdot (\gamma^+_{u^{-1}w_0}\bullet \pt_{w_0})\right)\\
&=\de_{w,u}\prod_{\al>0}(t-t^{-1}e^{-\al})\unit,
  \end{align*}
where the last equality follows from Part (1) of this theorem. 
\end{proof}

\begin{rem}There are similar dualities in \cite[Theorem 13 and Theorem 22]{LSZZ23}. Indeed, in loc.it., \[C_w=\gamma^+_w\odot\pt_e, \quad \widetilde{C}_w=\gamma_{w^{-1}w_0}^-\bullet \pt_{w_0}, \quad C^J_u=\frac{t_{w_J}}{P_J(t^2)}Y^\mul _J\bullet \gamma^+_{uw_J}\odot\pt_e, \quad w\in W, u\in W^J,\]
where \[P_J(t)=\sum_{v\in W_J}t_v\] is the Poincar\'e polynomial of $W_J$. Clearly for $K_T(G/B)$ there is no key difference between the definitions in \cite{LSZZ23} and in the present paper. But for $K_T(G/P_J)$, the class $\widetilde{C}^J_u$ in \cite[Definition 21]{LSZZ23} is not defined as simple as our $\gamma^+_{u^{-1}w_0}\bullet\pt_{w_0}$. Morover,   the proof in the present paper do not rely on the geometric property of motivic Chern classes, and is much simpler. 
\end{rem}

\section{Kazhdan-Lusztig basis for  Grassmannians}
In this section, we consider type $A$ Grassmannians, and use a result of Fan-Green to study properties of Kazhdan-Lusztig basis. We  assume $G$ is of type $A$ in this section.

\subsection{Temperley-Lieb algebra}
\begin{dfn}For each $w=s_{i_1}\cdots s_{i_k}$ and $I_w=(i_1,\cdots, i_k)$,  we define 
\[
\hat \gamma_{I_w}^\pm=\prod_{j=1}^k\gamma^\pm_{i_j}.
\]
\end{dfn}
Since the elements $\gamma_i^\pm$ do not satisfy the braid relations, this definition depends on the choice of the reduced sequence, and furthermore, $\hat\gamma_{I_w}^\pm\neq \gamma_w^\pm$ in general. Note that both $\gamma_w^\pm$ and $\hat \gamma_{I_w}^\pm$ are bases of $\bfH$ and of $Q_W^\mul$. 
\begin{example}Consider $A_2$ case, then all Kazhdan-Lusztig polynomials are equal to 1, and it is well-known that 
\[
\gamma_{s_1s_2s_1}^-=\sum_{w\le s_1s_1s_1}\ep_v\ep_{s_1s_2s_1}t_{s_1s_2s_1}^{-1}t_w\tau_w=\hat \gamma_{s_1s_2s_1}^--\gamma_{s_1}^-\neq \hat \gamma^-_{s_1s_2s_1}. 
\]
\end{example}

Denote  by $W_c$ the set of $iji$-avoiding  elements. They are also called  fully commutative elements.  Let  $\calI$ be the two-sided ideal of $\bfH$ generated by elements of the form
\[
\gamma^-_{w_J}
\]
where $W_J$ is a rank 2  parabolic of $W$ not isomorphic to $A_1\times A_1$. The quotient $\TL=\bfH/\calI$ is called the Temperley-Lieb (TL) algebra. Denote the projection by $p:\bfH\to \TL$. By \cite[Proposition 3.1.1]{FG97}, 
\begin{equation}
\label{eq:I}\calI=R\{\gamma_w^-~| ~w\in W\backslash W_c\}.
\end{equation}
So $\TL$ has a $R$-basis $p(\gamma^-_w), w\in W_c$\footnote{Note that in \cite{FG97} the algebra $\TL$ is defined to be $\bfH/\calI'$ where $\calI'$ is the two-sided ideal generated by $\gamma^+_{w_J}$ for all rank 2 parabolic of $W$ that are not isomorphic to $A_1\times A_1$. However, \cite[Theorem 3.8.2]{FG97} still holds in our case, and this  suffices for the present paper.}.

Equivalently, denoting $E_i=p(\gamma_i^-)$, then $\TL$ is generated by $E_1,...,E_n$ subject to relations 
\[
E_i^2=-(t+t^{-1})E_i, ~E_iE_{i\pm 1}E_i=E_i, ~E_iE_j=E_jE_i~ \text{ if } |i-j|>1.
\]
So $E_i$ satisfy the braid relations. Moreover, if $I_w$ is a reduced sequence of $w$, then $E_w:=p(\hat\gamma^-_{I_w}), w\in W_c$ form a basis of $\TL$. Note that for $w\in W_c$ which are $iji$-avoiding, they are fully commutative, so $\hat \gamma^-_{I_w}$ does not depend on the choice of reduced sequences of $w$. 

The following lemma shows that the Kazhdan-Lusztig basis for $w\in W_c$ becomes `multiplicatively' generated in the TL algebra. 	We are grateful to Rui Xiong for pointing out this to us. 
\begin{lem} \cite[Theorem 3.8.2]{FG97} \label{lem:FG}For any $w\in W_c$, fix a reduced sequence $I_w$, then  in $\TL$, 
\[p(\gamma_w^-)=E_w=p(\hat \gamma^-_{I_w}). \]
\end{lem}

\subsection{Kazhdan-Lusztig basis for Grassmannians}

In the remaining part of this section, we assume $J$ is a maximal proper subset of $\Pi$, in other words, $G/P_J$ is a Grassmannian. The following lemma is well-known. 

\begin{lem}\label{lem:comb}For any maximal proper subset $J\subset\Pi$, we have 
\[
W^J\subset W_c. 
\]
\end{lem}

\begin{thm}\label{thm:p}
Let $J$ be maximal proper in $\Pi$ . 
The canonical map of $R$-modules 
\[\rho:\bfH\twoheadrightarrow \bfH\gamma_{w_J}^+, \quad z\mapsto z\gamma_{w_J}^+\] induces a commutative diagram
\[
\xymatrix{\bfH\ar@{->>}[r]^-{\rho} \ar@{->>}[d]^-{p} &\bfH\gamma_{w_J}^+\subset \bfH\\
\TL\ar[ru]_-{\rho'} &}. 
\]
\end{thm}
\begin{proof} 
Note that \eqref{eq:I} says that  \[\ker p=\calI=R\{\gamma_w^-|w\not \in W_c\}.\] 
It suffices to prove that $\rho(\gamma_w^-)=0, \forall w\not\in W_c$. By Lemma \ref{lem:comb},  $w\not\in W_c$ implies that  $w\not\in W^J$  . Therefore, there exists $j\in J$ so that $ws_j<w$. Clearly,  $s_jw_J<w_J$, then  Lemma \ref{lem:KLfactor}  gives 
\[
\gamma_{w}^-=z_1\gamma_j^-, \quad \gamma^+_{w_J}=\gamma^+_jz_2, \quad z_1, z_2\in Q_W^\mul . 
\]
So 
\[\rho(\gamma_w^-)=\gamma_w^-\gamma_{w_J}^+=z_1\gamma_j^-\gamma_j^+z_2=z_10z_2=0.\]
\end{proof}
\begin{rem}The left $\bfH$-module $\bfH\gamma^+_{w_J}$ is nothing but the module considered by Deodhar \cite{D79} and Seorgel \cite{S97}. See \cite{LZZ20} for more details. 
\end{rem}

\begin{cor}\label{cor:p}Let $J\subset\Pi$ be maximal proper, then for any $u\in W^J$, we have $\gamma_u^-\gamma_{w_J}^+=\hat \gamma_{I_u}^-\gamma_{w_J}^+$. 
\end{cor}
\begin{proof}Note that for any $z\in \bfH$,  \[\gamma^+_{w_J}=\rho(z)=\rho'(p(z)),\] and Lemma \ref{lem:FG} gives $p(\gamma_u^-)=p(\hat\gamma_{I_u}^-)$, so \[\gamma_u^-\gamma_{w_J}^+=\rho'(p(\gamma^-_u))=\rho'(p(\hat\gamma^-_{I_w}))=\hat \gamma_{I_u}^-\gamma_{w_J}^+.\]
\end{proof}

\begin{assumption}\label{assump:comp}
 We say a set of sequences $I_w$ for $w\in W$ is  $J$-compatible if for any $w=uv$ with $u\in W^J, v\in W_J$, the reduced sequence $I_w$ is the concatenation of $I_u$ and $I_v$.
\end{assumption}
If the sequences $\{I_w, w\in W\}$ are $J$-compatible, then  
\[\hat \gamma_{I_{uv}}=\hat\gamma_{I_u}\hat\gamma_{I_v}, \quad u\in W^J, v\in W_J. \]

\begin{lem}\label{lem:b}Let $J\subset \Pi$ be maximal proper,  $w\in W$ and fix a set of $J$-compatible sequences $I_w, w\in W$. Write 
\begin{equation}\label{eq:b}
\de_w^\mul =\sum_{u\in W^J, v\in W_J}\hat b^\mul _{w,I_{uv}}\hat \gamma^-_{I_{uv}}, \quad \de^\mul _w=\sum_{u\in W^J, v\in W_J}b^\mul _{w,uv}\gamma_{uv}^-, \quad \hat b^\mul _{w,I_{uv}}, b^\mul _{w,uv}\in Q^\mul, 
\end{equation}
then 
\[
\hat b^\mul _{w,I_u}=b^\mul _{w,u}, \quad u\in W^J.
\]
\end{lem}
\begin{proof}  From Lemma \ref{lem:KLfactor}, if $v\neq e$, we know that $\gamma^-_{uv}=z\gamma^-_\al$ with $z\in Q_W^\mul , \al\in J$, so \[\gamma^-_{uv}\gamma^+_{w_J}=z\gamma_\al^-\gamma^+_{w_J}=0.\] Similarly, if $v\neq e$, since $\hat \gamma^-_{I_{uv}}=\hat\gamma^-_{I_u}\hat\gamma^-_{I_v}$, so \[\hat \gamma^-_{I_{uv}}\gamma^+_{w_J}=\hat \gamma^-_{I_{u}}\hat\gamma^-_{I_v}\gamma^+_{w_J}=\hat\gamma_{I_u}^-\cdot 0=0. \]Right multiplying $\gamma^+_{w_J}$ on the right of the identities in \eqref{eq:b}, we obtain
\[
\de_w^\mul \gamma^+_{w_J}=\sum_{u\in W^J}\hat b_{w,I_u}^\mul \hat \gamma^-_{I_u}\gamma^+_{w_J}, \quad \de^\mul _w\gamma^+_{w_J}=\sum_{u\in W^J}b^\mul _{w,u}\gamma_{u}^-\gamma^+_{w_J},
\]
By Corollary  \ref{cor:p},  we know that 
\[
\hat\gamma_{I_u}^- \gamma^+_{w_J}=\gamma_u^- \gamma^+_{w_J}, \quad \forall u\in W^J. 
\]
Therefore, $\hat b^\mul _{w,I_u}=b_{w,u}^\mul $.
\end{proof}

\begin{rem}
Note that to define $\hat b^\mul_{w,I_u}, u\in W^J$, one has to define $\hat \gamma^-_{I_{uv}}$ for all $u\in W^J, v\in W_J$, even though Lemma \ref{lem:comb} implies that  $u\in W^J$ are fully commutative. However,  Lemma \ref{lem:b}  implies that as long as the sequences $I_{uv}$ are J-compatible, $\hat b^\mul_{w,I_u}=b^\mul_{w,u}$ for any $w\in W, u\in W^J$. 
\end{rem}

\section{Hyperbolic KL-Schubert classes for Grassmannians}
In this section, we  define the KL-Schubert classes for the hyperbolic cohomology. Our definition is different from the previous one in \cite{LZZ20} and \cite{LSZZ23}, by swapping $\gamma^+_w$ by $\gamma^-_w$. We then establish the Poincar\'e duality, and then prove the Billey-type formula. 

\subsection{Hyperbolic Demazure algebra} We recall some notion from \cite{LZZ20}. Starting from now, let \[R=\bbZ[t, t^{-1}, (t+t^{-1})^{-1}].\]  It will not affect any result in the previous sections. Consider the hyperbolic formal group law over $R$:  \[F_\hy=\frac{x+y-xy}{1-(t+t^{-1})^{-2}xy}.\]
Denote $\mu=t+t^{-1}$ and $\mu_u=\mu^{\ell(u)}$. 

 Define the formal group algebra \[S^\hy=R[[x_\la|\la\in \La]]/\frakI,  \]  where $\frakI$ is the completion of the ideal generated by $x_0$ and \[ F_\hy(x_\la,x_\mu)-x_{\la+\mu},~ \la,\mu\in \La.\] This ring is the algebraic replacement of $\hy_T(\pt)$ where $\hy$ is the oriented cohomology corresponding to the hyperbolic formal group law, and $x_\la$ is the first hyperbolic Chern class of the line bundle over the classifying space determined by $\la$. One can then define $Q^\hy=Frac(S^\hy)$, $Q^\hy_W=Q^\hy\rtimes R[W]$ with basis $\de^\hy_w, w\in W$, and the divided difference operator $X_\al$ and push-pull operator $Y_\al$ for simple roots $\al$:  \[X_\al=\frac{1}{x_\al}(\de_\al^\hy-1), ~Y_\al=1+X_\al.\]
We have $X_\al^2=-X_\al$. These operators do NOT satisfy the braid relations. For each $w$, fixing a reduced sequence $I_w$, one then defines $X_{I_w}$ and $Y_{I_w}$. Denote $X_i=X_{\al_i}, Y_i=Y_{\al_i}$.  The subalgebra $\bfD_\hy$ of $Q_W^\hy$ generated by $X_{\al_i}, \al_i\in \Pi$ is called the formal Demazure algebra for the hyperbolic formal group law.

One can also define the dual $(Q_W^\hy)^*=\Hom(W,Q^\hy)$, which is a $Q^\hy$-module with basis $f_w, w\in W$ and componentwise product. The unit $\sum_wf_w$ is again denoted by $\unit$, which is the basis that spans $((Q^\hy_W)^*)^W$. The two actions $\bullet$ and $\odot$ can be similarly defined in this case, that is, $Q_W^\hy$ acts on $(Q_W^\hy)^*$ via the $\bullet$ and $\odot$ action. 

Similar as the multiplicative case, we define \[\pt_{e}=x_\Pi^\hh f_e,~ \pt_w=\de_w^\hy\bullet\pt_e=w(x_\Pi^\hy) f_w\in (Q_W^\hy)^*,  ~Y^\hy_J=\sum_{w\in W_J}\de_w^\hy \frac{1}{x^\hy _J}, \quad x^\hy _J=\prod_{\al<0, \al\in \Sigma_J}x_\al\in S^\hy. \]
If $J=\Pi$, denote $Y_J^\hy$ by $Y_\Pi^\hy$.  Again, the map $Y^\hh_\Pi\bullet \_~:(Q_W^\hy)^*\to ((Q_W^\hy)^*)^W\cong Q^\hy\unit$ is the algebraic model for the push-forward to the base point. Moreover, one can define the involution $\iota:Q^\hy_W\to Q^\hy_W$ so that the properties in \eqref{eq:inv} still holds. 

\subsection{KL-Schubert classes in hyperbolic cohomology}
There is a map of formal group laws  \[g:F_\hy(x,y)\to F_\mul(x,y)=x+y-xy,~ ~g(x)=\frac{(1-t^2)x}{x-(t^2+1)}, \] so that \[F_\mul(g(x), g(y))=g(F_\hy(x,y)).\] 
It induces a map
\[
\psi:S^\mul\to S^\hy, \quad f(x_\la)\mapsto f(g(x_\la)),~\forall f\in R[[x]],
\]
where we denote by $x_\la=1-e^{-\la}$ in the domain $S^\mul$, and also the same notation $x_\la$ in the target $S^\hy$. 
It also induces a ring homomorphism $\psi: Q^\mul\to  Q^\hy$, and hence a map of abelian groups:
\[
\psi:Q^\mul_W\to Q_W^\hy, ~a\de_w^\mul\mapsto \psi(a)\de_w^\hy,~ a\in Q^\mul. 
\]
It is easy to verify that 
\[
\psi(\tau_\al)=\mu Y_\al-t.
\]
so it induces a ring homomorphism:
\begin{equation}\label{eq:hyXY}\psi:Q_W^\mul\supset\bfH\to \bfD_h\subset Q_W^\hy, \quad \psi(\gamma^+_\al)= \mu Y_\al, \quad \psi(\gamma_\al^-)=\mu X_\al.\end{equation}
 In particular, we have 
\begin{equation}\label{eq:psigamma}\psi(\hat \gamma_{I_u}^-)=\mu_u X_{I_u}, \quad  u\in W, 
\end{equation}
where both of the two depend on the choice of reduced sequence. 
Since $\gamma_u^-, u\in W$ is a basis of $Q_W^\mul$, $\psi(\gamma_u^-)$ is a basis of $Q^\hy_W$. Let $\psi(\gamma_u^-)^*:=(\psi(\gamma_u^-))^*\in (Q_W^\hy)^*$ be the dual basis. Since $\gamma_u^-\neq \hat \gamma_{I_u}^-$ in general, so $\psi(\gamma_u^-)$ maybe different from $\mu_u X_{I_u}$. 
\begin{rem}
 One can also define the operators $X_\al$ and $Y_\al$ in $K$-theory, and they will permute the ideal  sheaf class and the structure sheaf class in $K_T(G/B)$, respectively. Based on this fact, hyperbolic case, \eqref{eq:hyXY}  shows that the two versions of Kazhdan-Lusztig bases actually correspond to ideal  sheaf class and structure sheaf class. 
\end{rem}

\begin{thm}\label{thm:hyper}
\begin{enumerate}
\item The following two classes are dual in $\hh_T(G/B)$: 
\[
\calC_w:= \mu_w^{-1}\psi(\gamma^{-}_w)\odot\pt_e, \quad \tilde\calC_u:=\mu_{u^{-1}w_0}^{-1}\psi(\gamma^+_{u^{-1}w_0})\bullet \pt_{w_0}. 
\]
That is, 
\[
Y^\hy_\Pi\bullet (\calC_w\cdot \tilde \calC_u)=\de_{w,u}\unit. 
\]
\item The following two classes are dual in $\hh_T(G/B)$. That is, 
\[
Y_\Pi^\hy\bullet (\calC_w \cdot \psi(\gamma_u^-)^*)=\de_{w,u}\mu_w^{-1}\unit.
\]
\end{enumerate}
\end{thm}
\begin{proof}
(1). It is easy to verify that  for any $\al>0$, 
\[
\psi(\frac{t-t^{-1}e^{-\al}}{1-e^{-\al}})=\psi(\frac{t-t^{-1}+t^{-1}x_\al}{x_\al})=\frac{t-t^{-1}+t^{-1}g(x_\al)}{g(x_\al)}=\frac{\mu}{x_\al}, 
\]
so 
\begin{equation}\label{eq:aw0}
\psi(a_{w_0})=\frac{\mu_{w_0}}{w_0x_\Pi^\hy}. 
\end{equation}
Therefore, we have
\begin{align*}
&Y_\Pi^\hy  \bullet\left(\mu_w^{-1}\psi(\gamma_w^-)\odot\pt_e\cdot \mu_u\mu^{-1}_{w_0}\psi(\gamma^+_{u^{-1}w_0})\bullet \pt_{w_0}\right)\\
&\overset{\eqref{eq:inv}}=\mu^{-1} _w\mu_u\mu^{-1}_{w_0}Y_\Pi^\hy  \bullet\left ((\iota(\psi(\gamma^-_{w}))\bullet\pt_e)\cdot  \psi(\gamma^+_{u^{-1}w_0})\bullet \pt_{w_0}\right)\\
&\overset{\eqref{eq:inv}}=\mu_w^{-1} \mu_u\mu^{-1}_{w_0}Y_\Pi^\hy  \bullet \left(\pt_e\cdot\psi (\gamma_w^-\bullet  \gamma^+_{u^{-1}w_0})\bullet \pt_{w_0}\right)\\
&=\mu_w^{-1} \mu_u\mu^{-1}_{w_0}Y_\Pi^\hy  \bullet \left(\pt_e\cdot \psi(\gamma_w^-\gamma^+_{u^{-1}w_0})\bullet\pt_{w_0}\right)\\
&\overset{\eqref{eq:Kprod}}=\mu_w^{-1} \mu_u\mu^{-1}_{w_0}Y_\Pi^\hy  \bullet \left(\pt_e\cdot \psi(\gamma_w^-\gamma^+_{u^{-1}w_0})|_{w_0}\de^\hy_{w_0}\bullet\pt_{w_0}\right)\\
&\overset{\text{Lem. }\ref{lem:KLcom}}=\de_{w,u}\mu_w^{-1} \mu_u\mu^{-1}_{w_0}Y^\hy  _\Pi\bullet \left(\pt_e\cdot \psi(a_{w_0})\de^\hy_{w_0}\bullet \pt_{w_0}\right)\\
&\overset{\eqref{eq:aw0}}=\de_{w,u}\mu_{w_0}^{-1}Y^\hy_\Pi\bullet\left(x_\Pi^\hy f_e\cdot \frac{\mu_{w_0}}{w_0x_\Pi^\hy}\de_{w_0}^\hy\bullet (w_0(x_\Pi^\hy)f_{w_0})  \right)\\
&\overset{\eqref{eq:action}}=\de_{w,u}\sum_{v\in W}\de_v^\hy \frac{1}{x_\Pi^\hy}\bullet(x_\Pi^\hy f_e\cdot f_e)\\
&=\de_{w,u}\unit. 
  \end{align*}

(2). Write 
\[
\psi(\gamma_w^-)=\sum_{v\in W}a_{w,v}\de_v^\hy, \quad \de_w^\hy=\sum_{v\in W}b_{w,v}\psi(\gamma_{v}^-),
\]
then the two matrices $(a_{w,v})_{w,v\in W}$ and $(b_{w,v})_{w,v\in W}$ are inverse to each other. Moreover, 
\begin{align*}
\psi(\gamma_u^-)^*&=\sum_{v_1}b_{v_1,u}f_{v_1}, \\
\calC_w&=\mu_w^{-1}\psi(\gamma_w^-)\odot\pt_e=\mu_w^{-1}\sum_{v}a_{w,v}\de_v^\hy\odot {x_\Pi^\hy}f_e\overset{\eqref{eq:action}}=\mu_w^{-1}\sum_va_{w,v}{vx_\Pi^\hy}f_v. 
\end{align*}
We have
\begin{align*}
Y^\hy_\Pi\bullet \left(\calC_w\cdot \psi(\gamma_u^-)^*\right)&=Y_\Pi^\hy\bullet \left(\mu_w^{-1}(\sum_{v}a_{w,v}{vx_\Pi^\hy}f_v)\cdot (\sum_{v_1}b_{v_1,u}f_{v_1})\right)\\
&\overset{\eqref{eq:Kprod}}=\mu_w^{-1}\sum_{v_2\in W}\de_{v_2}^{\hy}\frac{1}{x_\Pi^\hy}\bullet  \left(\sum_{v}a_{w,v}b_{v,u}{vx_\Pi^\hy}f_v\right)\\
&\overset{\eqref{eq:action}}=\mu_w^{-1}\sum_{v_2\in W}(\sum_{v}a_{w,v}b_{v,u})f_{vv_2}\\
&=\mu_w^{-1}\de_{w,u}\unit. 
\end{align*}
\end{proof}

\begin{cor}\label{cor:dual} We have the following equality of classes:\[
\mu_w\psi(\gamma_w^-)^* =\tilde \calC_w
\]
\end{cor}

\begin{rem}Note that in the present paper we work with classes in $(Q^\hy_W)^*\cong Q^\hy\otimes_{S^\hy}\hy_T(G/B)$, that is, the localized classes, but all the classes we consider in this paper are actual classes without localization. The same comment applies to the $K$-theory classes in the previous sections. 
\end{rem}
\subsection{KL-Schubert classes for  $G/P_J$}
From now on, we assume $G/P_J$ is a Grassmannian.  Recall from \cite[Theorem 5.4]{LZZ20} \[\psi(\gamma^+_{w_J})=\mu_{w_J}Y_J ^\hy.  \]

\begin{dfn}For any $u, w\in W^J$, define the following KL-Schubert class $\calC_w^J$ and opposite KL-Schubert class $\tilde \calC_u^J$ for $\hh_T(G/P_J)$:
\begin{align*}
\calC_w^J&:=\mu_{w}^{-1}Y_J^\hy\bullet \psi(\gamma^{-}_w)\odot\pt_e=\mu_w^{-1}\psi(\gamma_w^-)\odot Y_J^\hy\bullet \pt_e, \\
 \tilde \calC^J_u&:=\tilde\calC_u=\mu_{u^{-1}w_0}^{-1}\psi(\gamma^+_{u^{-1}w_0})\bullet\pt_{w_0}.
\end{align*}
\end{dfn}
By definition, since $Y_J^\hy\bullet (Q^\hy_W)^*\subset ((Q^\hy_W)^*)^{W_J}$, so  $\calC^J_w$ is $W_J$-invariant. Corresponding fact for the opposite KL-Schubert classes is proved below.

\begin{lem}\label{lem:WJ} \begin{enumerate}
\item For any $w\in W^J$, let $I_w$ be a reduced sequence of $w$, then we have \[\calC_w^J= Y_J^\hy \bullet (X_{I_w}\odot\pt_e).\]
That is, $\calC_w^J$ coincides with the Bott-Samelson classes. 
\item Let $u\in W^J$,   the class $ \tilde \calC^J_u$ is $W^J$-invariant. 

\end{enumerate}
\end{lem}
\begin{proof} (1). 
We have 
\[
\psi(\gamma_w^-)Y_J^\hy =\mu_{w_J}^{-1}\psi(\gamma^-_w\gamma_{w_J}^+)\overset{\text{Cor. }\ref{cor:p}}=\mu_{w_J}^{-1}\psi(\hat\gamma^-_{I_w}\gamma_{w_J}^+)=\mu_{w_J}^{-1}\psi(\hat\gamma^-_{I_w})\psi(\gamma^+_{w_J})\overset{\eqref{eq:psigamma}}=\mu_w X_{I_w}Y_J^\hy . 
\]
So 
\[\calC_w^J=\mu_w^{-1 } Y_J^\hy \bullet \psi(\gamma^{-}_w)\odot\pt_e\overset{\eqref{eq:YJcom}}=\mu_w^{-1}\psi(\gamma_w^{-})Y_J^\hy \odot\pt_e\overset{\eqref{eq:psigamma}}=X_{I_w}Y_J^\hy \odot\pt_e=Y_J^\hy \bullet (X_{I_w}\odot\pt_e). \]
Note that in this part we used the property that the two actions $\odot$ and $\bullet$ commute with each other. 

(2). The proof is similar as that of Theorem \ref{thm:Kdual}.(2). From \[\ell(u^{-1}w_0)=\ell(w_J)+\ell(w_Ju^{-1}w_0),\] we see that  Lemma \ref{lem:KLfactor} implies that one can write $\gamma^+_{u^{-1}w_0}=\gamma^+_{w_J}z$ for some $z\in Q_W^\mul$. So \[\tilde \calC^J_u=\mu^{-1} _{u^{-1 w_0}} \psi(\gamma_{w_J}^+)\psi(z)\bullet \pt_{w_0}=\mu_{u^{-1}w_0}^{-1}\mu_{w_J}Y_J^\hy \bullet (\psi(z)\bullet \pt_{w_0})\in ((Q^\hy_W)^*)^{W_J}, \] that is, it is $W_J$-invariant.
\end{proof}
The following duality for Grassmannians can be proved similarly as Theorem \ref{thm:Kdual}.(2). 

\begin{thm}\label{thm:hyperGP}
The classes $\calC_w^J, w\in W^J$ are dual to the classes 
$
\tilde  \calC_u^J, u\in W^J$ via the map 
\[Y_{\Pi/J}^\hy \bullet\_: ((Q_W^\hy)^*)^{W_J}\to ((Q_W^\hy)^*)^W\cong Q^\hy\unit. \]
\end{thm}

\section{The Billey-type formula} 
In this section we study the restriction of the opposite KL-Schubert classes for Grassmannians, and deduce the Billey-type formula. This was the original motivation of this paper.

\subsection{Restrictions of the opposite KL-Schubert classes} 
Recall that since both $\{\gamma^-_u, u\in W\}$ and $\{\hat\gamma^-_u, u\in W\}$ are bases of $Q_W^\mul$, we can write
\begin{equation}\label{eq:de}
\de_w^\mul=\sum_{u\le w}b^\mul_{w,u}\gamma_{u}^{-}=\sum_{u\le w}\hat b^\mul_{w,I_u}\hat \gamma^-_{I_u}\in Q_W^\mul. 
\end{equation}
For similar reason,  in hyperbolic cohomology, we can write
\begin{equation}\label{eq:hyb}
\de^\hy_w=\sum_{u\in W}b^\hy_{w,u}\psi(\gamma_u^-)=\sum_{u\in W}\hat b_{w,I_u}^\hy\psi(\hat\gamma_{I_u}^-)\in Q_W^\hy, ~w,u\in W. 
\end{equation}
Applying $\psi$ to \eqref{eq:de} obtains
\[\de_w^\hy=\psi(\de_w^\mul)=\sum_{u\in W}\psi(b^\mul_{w,u})\psi(\gamma_{u}^{-})=\sum_{u\in W }\psi(\hat b^\mul_{w,I_u})\psi(\hat \gamma^-_{I_u}).\]
Therefore, 
\begin{equation}\label{eq:compb}b^\hy_{w,u}=\psi(b^\mul_{w,u}), \quad \hat b^\hy_{w,I_u}=\psi(\hat b^\mul_{w,I_u}).\end{equation}

Note that by Corollary \ref{cor:dual}, we have
\begin{equation}\label{eq:c}
\tilde \calC_u^J=\mu_u(\psi(\gamma_u^-))^*=\mu_u\sum_{w\ge u}b^\hy_{w,u}f_w, ~u\in W^J.
\end{equation}
\begin{lem}  We have  \[b^\hy_{wv,u}=b^\hy_{w,u}, \forall u\in W^J, w\in W, v\in W_J. \]
\end{lem}
\begin{proof}
Note from \cite{CZZ19} we know that $((Q_W^\hy)^*)^{W_J}$ has a basis \[g_w:=\sum_{v\in W_J}f_{wv}, w\in W^J,\] and Lemma \ref{lem:WJ} implies that $\tilde\calC_{u}^J, u\in W^J$ belongs to $((Q_W^\hy)^*)^{W_J}$, therefore, \[b^\hy_{wv, u}=b^\hy_{w,u}, \forall u\in W^J, w\in W, v\in W_J.\]
\end{proof}


\begin{thm}\label{thm:Billey}Assume $\{I_w\}$  is J-compatible (Assumption \ref{assump:comp}).  For $w\in W, u\in W^J$, we have
\[
\tilde \calC^J_u|_w=\mu_u\hat b^\hy_{w,I_u}. 
\]
\end{thm}
\begin{proof}
Applying Lemma \ref{lem:b}, we have
\begin{equation}\label{eq:bhy}
b^\hy_{w,u}=\hat b^\hy_{w,I_u}, ~\forall w\in W, u\in W^J.
\end{equation}
the conclusion then follows from \eqref{eq:compb} and  \eqref{eq:c}. 
\end{proof}

\subsection{The Billey-type formula}
It follows from Theorem \ref{thm:Billey} that to compute the restriction $\tilde \calC^J_u|_w, u\in W^J, w\in W $, it suffices to compute $\mu_u 
\hat b^{\hy}_{w,I_u}$. 
Recall \eqref{eq:psigamma} that $\psi(\hat \gamma_{I_u}^-)=\mu_u X_{I_u}$, so to compute $\hat b_{w,I_u}^\hy$, it suffices to look at the Billey-type formula associated to the operator $X^\hy_\al$, which was studied in \cite{LZ17}. We recall the setup. 

Since $X_\al=\frac{1}{x_\al}(\de_\al^\hy-1)$, so $\de_\al^\hy=1+x_\al X_\al$. For simplicity, denote 
\[
x_i=x_{\al_i}, ~\de_i=\de_{s_i}, ~X_i=X_{\al_i}, ~X_{ij}=X_iX_j. 
\]Let $w=s_{i_1}\cdots s_{i_k}$, $\be_j=s_{i_{1}}\cdots s_{i_{j-1}}\al_{i_j}$. Define the root polynomials
\[
h_{j}(\be)=1+x_\be X_j, \quad R_{I_w}=\prod_{j}h_{i_j}(\be_j). 
\]
The essence of the Billey-type formula is to  assume that $x_\be\in S^\hy$ commute with $X_\al$. For more details, see \cite{LZ17}. 

Expand 
\[
R_{I_w}=\sum_{v\le w}\hat b'^\hy _{I_w,I_v}X_{I_v}=\sum_v\hat b'^\hy_{I_w,I_v}\mu_v^{-1}\psi(\gamma_v^-), 
\]
then it follows from \cite[Theorem 4.4]{LZ17} that 
\[
\hat b'^\hy_{I_w,I_u}=\hat b_{w,I_u}^\hy, \quad u,w\in W
\] 
Consequently,  the definition of $R_{I_w}$ does not depend on the choice of $I_w$. So we will write $\hat b'^\hy _{w,I_u}$ for $\hat b'^\hy _{I_w,I_u}$. 
Consequently,   we have  \[\tilde \calC^J_u|_w=\mu_u\hat b'^\hy_{w,I_u}, ~u\in W^J, w\in W.\]

Before presenting an example of computation, 
recall  the twisted braid relations from \cite{L16}:
\begin{align}\label{eq:braid}
\begin{cases}X_{j}X_iX_j-X_{i}X_jX_i=\mu^{-2}(X_j-X_i), & |i-j|=1;\\
X_iX_j=X_jX_i, & |i-j|>1. 
\end{cases}
\end{align}
When computing $\hat b'^\hy_{w,I_u}$, wee will need these relations, together with $X_i^2=-X_i$.

\begin{example}
	Consider the $A_4$ case with $J = \{1,2,4\}$. We fix  $J$-compatible sequences that include $s_1s_2s_1, s_2s_3s_2, s_4s_3s_4$. Let 
\[w=s_2s_1s_3s_2s_4s_3,~ u = s_2s_3\in W^J. \] Then the eight subwords in the table below will give $X_{2}X_3$, so   $\hat b'^\hy_{w,I_u}$ is the sum of the elements in the second column of this table: 
	\[
	\begin{array}{c|c}
\text{subwords} & \\
\hline
(s_2,-,s_3,-,-,-)&x_2x_{2+3}\\
(s_2,-,-,-,-,s_3)&x_2x_{1+2+3+4}\\
(-,-,-,s_2,-,s_3)&x_{1+2+3}x_{1+2+3+4}\\
(s_2,-,-,s_2,-,s_3)&-x_2x_{1+2+3}x_{1+2+3+4}\\
(s_2,-,s_3,-,-,s_3)&-x_2x_{2+3}x_{1+2+3+4}\\
(s_2,-,s_3,s_2,-,s_3)&\sharp_1:\mu^{-2}x_2x_{2+3}x_{1+2+3}x_{1+2+3+4}\\
(s_2,s_1,-,s_2,-,s_3)&\sharp_2: \mu^{-2}x_2x_{1+2}x_{1+2+3}x_{1+2+3+4}\\
(s_2,-,s_3,-,s_4,s_3)& \sharp_3:\mu^{-2}x_{2}x_{2+3}x_{2+3+4}x_{1+2+3+4}
	\end{array}
	\]

We explain the computation. We need to compute the  terms that give $X_2X_3$. The root polynomial is
\[
R_w=(1+x_{2}X_2)(1+x_{1+2}X_1)(1+x_{2+3}X_3)(1+x_{1+2+3}X_2)(1+x_{2+3+4}X_4)(1+x_{1+2+3+4}X_3).
\]
The coefficients for the first three subwords in the table above can be obtained easily, and the fourth and the fifth ones come from the relation $X_i^2=-X_i$. 

The subword $(s_2,-,s_3,s_2,-,s_3)$ gives 
\[
x_2x_{2+3}x_{1+2+3}x_{1+2+3+4}X_{2}X_3X_2X_3,
\]
and 
\[
X_{2}X_3X_2X_3\overset{\eqref{eq:braid}}=X_2(X_2X_3X_2-\mu^{-2}X_2+\mu^{-2}X_3)=-X_2X_3X_2+\mu^{-2}X_2+\mu^{-2}X_2X_3. 
\]
The last term multiplying with $x_2x_{2+3}x_{1+2+3}x_{1+2+3+4}$ gives  the term $\sharp_1$. 

Similarly,  the subword $(s_2,s_1,-,s_2,-,s_3)$ gives
\[
x_2x_{1+2}x_{1+2+3}x_{x1+2+3+4}X_2X_1X_2X_3, 
\]
and 
\begin{align*}
X_2X_1X_2X_3&\overset{\eqref{eq:braid}}=(X_1X_2X_1-\mu^{-2}X_1+\mu^{-2}X_2)X_3\\
&=X_1X_2X_1X_3-\mu^{-2}X_1X_3+\mu^{-2}X_2X_3\\
&=X_1X_2X_3X_1-\mu^{-2}X_1X_3+\mu^{-2}X_2X_3. 
\end{align*}
The last term multiplying with $x_2x_{1+2}x_{1+2+3}x_{x1+2+3+4}$ gives the term $\sharp_2$.

Consider the subword $(s_2,-,s_3,-,s_4,s_3)$, which gives 
\[
x_{2}x_{2+3}x_{2+3+4}x_{1+2+3+4}X_2X_3X_4X_3.
\]
We have 
\begin{align*}
X_2X_3X_4X_3&=X_2(X_4X_3X_4-\mu^{-2}X_4+\mu^{-2}X_3)=X_2X_4X_3X_4-\mu^{-2}X_2X_4+\mu^{-2}X_2X_3.
\end{align*}
This gives the coefficient in $\sharp_3$.
\end{example}

\end{document}